\documentclass[12pt,a4paper]{article}
\usepackage{mathrsfs}
\usepackage{epsf, graphicx}
\usepackage{latexsym,amsfonts,amsbsy,amssymb}
\usepackage{amsmath,amsthm}
\usepackage{color}
\makeatletter 

\@addtoreset{equation}{section}
\makeatother 
\allowdisplaybreaks 

\newtheorem{theorem}{Theorem}[section]

\newtheorem{corollary}{Corollary}[section]
\newtheorem{remark}{Remark}[section]

\begin{document}
\title{Lower Bounds of the Discretization for Piecewise Polynomials
\footnote{This project was supported by the National Natural Science
Foundation of China (11001259, 11031006, 2011CB309703).}
}
\author{Qun Lin\footnote{LSEC, Institute of Computational Mathematics,
 Academy of Mathematics and Systems Science, Chinese Academy of Sciences, Beijing 100190,
China (linq@lsec.cc.ac.cn).}\ ,\  \   Hehu Xie\footnote{LSEC,
Institute of Computational Mathematics, Academy of Mathematics and
Systems Science, Chinese Academy of Sciences, Beijing 100190,
China (hhxie@lsec.cc.ac.cn).}\ \ \ and \ \
Jinchao Xu\footnote{Center for Computational Mathematics and
  Applications and Department of Mathematics,
Pennsylvania State University, University Park, PA 16802, USA
(xu@math.psu.edu)
The work of this author was partially supported by US National Science
Foundation through DMS 0915153 and DMS 0749202}}
\date{}
\maketitle

\begin{abstract}
  Assume that $V_h$ is a space of piecewise polynomials of degree less
  than $r\geq 1$ on a family of quasi-uniform triangulation of size
  $h$.  Then the following well-known upper bound holds for a
  sufficiently smooth function $u$ and $p\in [1, \infty]$
$$
\inf_{v_h\in V_h}\|u-v_h\|_{j,p,\Omega,h} \le C h^{r-j}
|u|_{r,p,\Omega},\quad 0\le j\le r.
$$
In this paper, we prove that, roughly speaking, if $u\not\in V_h$, the
above estimate is sharp. Namely,
$$
    \inf_{v_h\in V_h}\|u-v_h\|_{j,p,\Omega,h} \ge c h^{r-j},\quad 0\le j\le r, \ \ 1\leq p\leq \infty,
$$
for some $c>0$.

The above result is further extended to various situations including
more general Sobolev space norms, general shape regular grids and many
different types of finite element spaces.  As an application, the
sharpness of finite element approximation of elliptic problems and the
corresponding eigenvalue problems is established.


\end{abstract}
\section{Introduction}
Error analyses for many numerical methods are mostly presented for
upper bound estimates of the approximation error.  This paper is
devoted to lower bound error estimate and its applications for
piecewise polynomial approximation in Sobolev spaces.  Our work was
inspired by some recent studies of lower end approximation of
eigenvalues by finite element discretization for some elliptic partial
differential operators \cite{HuHuangLin,LinXieLuoLi}.  One crucial
technial ingredient that is needed in the analysis in
\cite{HuHuangLin,LinXieLuoLi} is some lower bound of the eigenfunction
discretization error by the finite element method.

Lower bound error estimates have been studied in the literature for
some special cases.  In Babu\v{s}ka and Miller
\cite{BabuskaMiller,BabuskaStroulis}, lower bounds of the
discretization error were obtained for second order elliptic problem
by a bilinear element discretization by the Taylor expansion method
under the assumption that the solution is smooth enough on the
prescribed domain.  More recently in K\v{r}\'{i}\v{z}ek, Roos, and
Chen in \cite{KrizekRoosChen} (which partially inspired the work in
this paper), two-sided bounds were obtained for the discretization
error of linear and bilinear elements on the uniform meshes by
superconvergence theory and interpolation error estimate.

The aim of this paper is to derive lower bound results of the error
by piecewise polynomial approximation for much more general classes
of problems under much weaker and more natural assumptions on grids
and smoothness of functions to be approximated.  As a special
application, lower bounds of the discretization error by a variety
of finite element spaces can be easily obtained.  For example, the
following lower error bounds (see Sections 3 and 4) are valid for
finite element (consisting of piecewise polynomials of degree less
than $r$) approximation to $2m$-th order elliptic boundary value
problems:
\begin{equation*}
\|u-u_h\|_{j,p,h} \geq  C h^{r-j},\ \ \ \ 0\leq j\leq r,
\end{equation*}
where the positive constant $C$ is independent of the mesh size $h$.
This kind of results plays a very important in the analysis of lower
end eigenvalue approximations in \cite{HuHuangLin,LinXieLuoLi}.

The outline of the rest of the paper is as follows. Section 2 is
devoted to general derivation of lower bounds of the error by
piecewise polynomial approximation.  Section 3 is for lower bounds of
the discretization error of the second order elliptic problem and the
corresponding eigenpair problem by finite element method.  Section
concerns with a generalization of the results from Section 3 to the $2m$-th
order elliptic problem and the corresponding eigenpair problem.
Section 5 contains some brief concluding remarks.


\section{Notation and basic results}
In this section, we first introduce the used notation and then state some lower bound results of the
piecewise polynomial approximation error which is a basic tool in this paper.

Here we assume that $\Omega\subset \mathcal{R}^n$ ($n\geq 1$) is a bounded polytopic domain with Lipschitz
continuous boundary $\partial\Omega$.  Throughout this paper, we use the standard notation for the usual Sobolev
spaces and the corresponding norms, semi-norms, and inner products as in \cite{BrennerScott, Ciarlet}.
Let us introduce the multi-index notation. A multi-index $\alpha$ is an $n$-tuple of non-negative
integers $\alpha_i$. The length of $\alpha$ is given by $$|\alpha|=\sum_{i=1}^n\alpha_i.$$
The derivative $D^{\alpha}v$ is then defined by
$$D^{\alpha}v=\Big(\frac{\partial}{\partial x_1}\Big)^{\alpha_1}\cdots \Big(\frac{\partial}{\partial x_n}\Big)^{\alpha_n}v.$$
For a subdomain $G$ of $\Omega$, the usual Sobolev spaces $W^{m,p}(G)$  with norm $\|\cdot\|_{m,p,G}$
and semi-norm $|\cdot|_{m,p,G}$ are used. In the case $p=2$, we have $H^m(G)=W^{m,2}(G)$ and the index
$p$ will be omitted. The $L^2$-inner product on $G$ is denoted by $(\cdot,\cdot)_G$. For $G\subset \Omega$
 we write $G\subset\subset \Omega$ to indicate that ${\rm dist}(\partial\Omega, G)>0$ and ${\rm meas}(G)>0$.

We introduce a face-to-face partition $\mathcal{T}_h$ of the
computational domain $\Omega$ into elements $K$ (triangles, rectangles,
tetrahedrons, bricks, etc.) such that
$$\bar{\Omega}=\bigcup_{K\in\mathcal{T}_h}K$$
and let $\mathcal{E}_h$ denote  a set of all $(n-1)$-dimensional facets of all elements $K\in\mathcal{T}_h$.
Here $h:=\max\limits_{K\in\mathcal{T}_h} h_K$ and $h_K=\text{diam}\ K$
denote the global and local mesh size, respectively \cite{BrennerScott,Ciarlet}.
We also define $\mathcal{T}_h^G=\big\{K\in \mathcal{T}_h\ {\rm and}\ K\subset G\big\}$ and
$h_G=\max\limits_{K\in \mathcal{T}_h^G} h_K$.
A family of partitions $\mathcal{T}_h$ is said to be {\it regular} if it satisfies the following condition:
$$\exists \sigma>0 \makebox{ such that} \ \ {h_K}/{\tau_K}>\sigma \ \ \ \forall K\in \mathcal{T}_h, $$
where $\tau_K$ is maximum diameter of the inscribed ball in $K\in\mathcal{T}_h$. A regular family of
partitions $\mathcal{T}_h$ is called {\it quasi-uniform} if it satisfies
$$\exists \beta>0\ \ \makebox{
such that}\ \ \max\{{h}/{h_K}, \ K\in \mathcal{T}_h\}\leq \beta.$$

Based on the partition $\mathcal{T}_h$, we build the finite element space $V_h$ of piecewise polynomial functions of degree less than $r$ (see \cite{BrennerScott,Ciarlet}).  In order to perform the error analysis,
  we define the following piecewise type semi-norm for $v\in W^{j,p}(G)\cup V_h$ with $G\subseteq \Omega$
\begin{eqnarray*}
|v|_{j,p,G,h}&:=&\left(\sum_{K\in\mathcal{T}_h^G}\int_{K}\sum_{|\alpha|=j}|D^{\alpha}v|^pdK\right)^{\frac{1}{p}},
\ \ \ \ 1\leq p<\infty,
\end{eqnarray*}
and
\begin{eqnarray*}
|v|_{j,\infty,G,h}&:=&\max_{K\in \mathcal{T}_h^G}|v|_{j,\infty,K}.
\end{eqnarray*}
We will drop $G$ when $G=\Omega$. Throughout this paper, the symbol
$C$ (with or without subscript) stands for a positive generic constant which may attain different
values at its different occurrences and which is independent of the mesh size $h$, but may depend
on the exact solution $u$.


\begin{theorem}\label{Lower_Bound_Theorem_Approximation_1}
Assume $u\in W^{r+\delta,p}(G)$\ $(\delta>0)$ and that there exists a multi-index $\gamma$ with $|\gamma|=r$
such that $\|D^{\gamma}u\|_{0,p,G}>0$ and $D^{\gamma}v_h=0$ for any $v_h\in V_h$.
Then the following lower bound of the approximation error holds when the family $\{\mathcal{T}_h\}$ of partitions
   is quasi-uniform
\begin{equation}\label{lower_convergence_0_1}
\inf_{v_h\in V_h}\|u-v_h\|_{j,p,G,h}\geq C_1h^{r-j}, \ \ \ \ 0\leq j\leq r,
\end{equation}
where $1\leq p\leq \infty$ and $C_1$ is dependent on $u$.
\end{theorem}
\begin{proof}
We prove the result (\ref{lower_convergence_0_1}) by a reduction process. The assumption
that (\ref{lower_convergence_0_1}) is not correct
means that for an arbitrarily small $\varepsilon>0$ there exist small enough $h$ and  $v_h\in V_h$ such that
  $h^{\delta}\|u\|_{r+\delta,p,G}<\varepsilon$ and
\begin{eqnarray}\label{assumption_reduction_1}
\frac{\|u-v_h\|_{j,p,G,h}}{h^{r-j}}&<&\varepsilon.
\end{eqnarray}
Now we show that (\ref{assumption_reduction_1}) leads to a contradiction.

Combining $u\in W^{r+\delta,p}(G)$, (\ref{assumption_reduction_1}), the quasi-uniform property of
$\mathcal{T}_h$ and the inverse inequality for finite element functions, we have
\begin{eqnarray}\label{Error_H_2}
|u-v_h|_{r,p,G,h}&\leq& \|u-\Pi_h^{r}u\|_{r,p,G,h}+\|\Pi_h^{r}u-v_h\|_{r,p,G,h}\nonumber\\
&\leq&C_2h^{\delta}\|u\|_{r+\delta,p,G}+ C_3h^{j-r}\|\Pi_h^{r}u-v_h\|_{j,p,G,h}\nonumber\\
&\leq&C_2h^{\delta}\|u\|_{r+\delta,p,G}+ C_3h^{j-r}\|\Pi_h^{r}u-u\|_{j,p,G,h}\nonumber\\
&&+C_3h^{j-r}\|u-v_h\|_{j,p,G,h}\nonumber\\
&\leq& (C_2+C_3C_4)h^{\delta}\|u\|_{r+\delta,p,G} +C_2\varepsilon\nonumber\\
&\leq&(C_2+C_3C_4+C_3)\varepsilon,
\end{eqnarray}
where $\Pi_h^{r}u$ denotes a piecewise $r$ degree  polynomial interpolant of $u$ (discontinuous
or continuous) such that
\begin{eqnarray*}
\|u-\Pi_h^{r}u\|_{\ell,p,G,h}&\leq &C_4h^{r+\delta-\ell}\|u\|_{r+\delta,p,G},\ \ \ \ 0\leq \ell\leq r.
\end{eqnarray*}
Then the condition $D^{\gamma}v_h=0$ leads to
\begin{eqnarray*}
\|D^{\gamma}u\|_{0,p,G}=\|D^{\gamma} (u-v_h)\|_{0,p,G,h}\leq |u-v_h|_{r,p,G,h} &\leq &C\varepsilon,
\end{eqnarray*}
where $C=C_2+C_3C_4+C_3$.
This contradicts the inequality $\|D^{\gamma}u\|_{0,p,G}>0$ and thus the assumption
 (\ref{assumption_reduction_1}) is not true. Therefore, the lower bound
result (\ref{lower_convergence_0_1}) holds and the proof is complete.
\end{proof}

The result in Theorem \ref{Lower_Bound_Theorem_Approximation_1} can be extended to a regular
 family partitions and more general Sobolev space norms.

\begin{theorem}\label{Lower_Bound_Corollary_Approximation_1}
Assume $u\in W^{r+\delta,p}(G)$\ $(\delta>0)$ and that there exists a multi-index $\gamma$ with $|\gamma|=r$ such that $\|D^{\gamma}u\|_{0,p,G}>0$ and $D^{\gamma}v_h=0$ for any $v_h\in V_h$.
Then we have the following lower bound of the approximation error when the family $\{\mathcal{T}_h\}$ of
partitions  is regular
\begin{eqnarray}\label{lower_convergence_0_2}
\inf_{v_h\in V_h}\left(\sum_{K\in\mathcal{T}_h^G}h_K^{p(j-r)}
\big\|u-v_h\big\|_{j,p,K}^p\right)^{\frac{1}{p}}\geq C_5, \ \ \ \ 0\leq j\leq r,
\end{eqnarray}
and
\begin{eqnarray}\label{lower_convergence_0_3}
\inf_{v_h\in V_h}\left(\sum_{K\in\mathcal{T}_h^G}h_K^{p\big((j-r)+n(\frac{1}{p}-\frac{1}{q})\big)}
\big\|u-v_h\big\|_{j,q,K}^p\right)^{\frac{1}{p}}\geq C_6, \ \ \ \ 0\leq j\leq r,
\end{eqnarray}
where $1\leq p<\infty$, $1\leq q\leq \infty$ $(W^{r+\delta,p}(G)$ can be imbedded into $W^{j,q}(G))$,
 $C_5$ and $C_6$ are positive constants independent of mesh size $h_G$, but dependent on $u$.
\end{theorem}
\begin{proof}
We only prove the result (\ref{lower_convergence_0_3}), since (\ref{lower_convergence_0_2}) can be deduced
directly from (\ref{lower_convergence_0_3}).

To get (\ref{lower_convergence_0_3}), we will use a similar reduction as in the proof of
Theorem \ref{Lower_Bound_Theorem_Approximation_1}.
The assumption that (\ref{lower_convergence_0_3}) is not correct
 means that for an arbitrarily small $\varepsilon>0$,
there exist $v_h\in V_h$ and small enough $h_G$ such that
$h_G^{\delta}\|u\|_{r+\delta,p,G}<\varepsilon$ and
\begin{eqnarray}\label{assumption_reduction_2}
\left(\sum_{K\in\mathcal{T}_h^G}h_K^{p\big((j-r)+n(\frac{1}{p}-\frac{1}{q})\big)}
\big\|u-v_h\big\|_{j,q,K}^p\right)^{\frac{1}{p}}&<&\varepsilon.
\end{eqnarray}
We will show this assumption leads to a contradiction.
Combining $u\in W^{r+\delta,p}(G)$, (\ref{assumption_reduction_2}), and the inverse inequality
for piecewise polynomial functions, we have
\begin{eqnarray}\label{Error_H_l}
|u-v_h|_{r,p,G,h}&\leq& \|u-\Pi_h^{r}u\|_{r,p,G,h}+\|\Pi_h^{r}u-v_h\|_{r,p,G,h}\nonumber\\
&\leq&C_7h_G^{\delta}\|u\|_{r+\delta,p,G}+ C_8 \left(\sum_{K\in\mathcal{T}_h^G}h_K^{p\big((j-r)+n(\frac{1}{p}-\frac{1}{q})\big)}
\big\|\Pi_h^{r}u-v_h\big\|_{j,q,K}^p\right)^{\frac{1}{p}}\nonumber\\
&\leq&C_7h_G^{\delta}\|u\|_{r+\delta,p,G}
+C_8\left(\sum_{K\in\mathcal{T}_h^G}h_K^{p\big((j-r)+n(\frac{1}{p}-\frac{1}{q})\big)}
\big\|u-\Pi_h^{r}u\big\|_{j,q,K}^p\right)^{\frac{1}{p}}\nonumber\\
&&\ \ +C_8\left(\sum_{K\in\mathcal{T}_h^G}h_K^{p\big((j-r)+n(\frac{1}{p}-\frac{1}{q})\big)}
\big\|u-v_h\big\|_{j,q,K}^p\right)^{\frac{1}{p}}\nonumber\\
&\leq& (C_7+C_8C_9)h_{G}^{\delta}\|u\|_{r+\delta,p,G} +C_8\varepsilon\nonumber\\
&\leq&(C_7+C_8C_9+C_8)\varepsilon,
\end{eqnarray}
where $\Pi_h^{r}u$ denotes a piecewise $r$ degree polynomial interpolant of $u$
(discontinuous or continuous) for which we have the following error estimate \cite{BrennerScott,Ciarlet}
\begin{eqnarray*}
\|u-\Pi_h^{r}u\|_{\ell,q,K}&\leq &C_9h_K^{r+\delta-\ell+n\big(\frac{1}{q}-\frac{1}{p}\big)}
\|u\|_{r+\delta,p,K},\ \ \ \ 0\leq \ell\leq r,\ \ \forall K\in\mathcal{T}_h.
\end{eqnarray*}
Then combining (\ref{Error_H_l}) and the condition $D^{\gamma}v_h=0$ ($|\gamma|=r$) leads to
\begin{eqnarray*}
\|D^{\gamma}u\|_{0,p,G}=\|D^{\gamma} (u-v_h)\|_{0,p,G,h}\leq |u-v_h|_{r,p,G,h} &\leq &C\varepsilon,
\end{eqnarray*}
where $C=C_7+C_8C_9+C_8$.  This contradicts the condition $|D^{\gamma}u|_{0,p,G}>0$ and thus the
 assumption (\ref{assumption_reduction_2})  is not true.
 Hence the lower bound result (\ref{lower_convergence_0_3}) holds
 and the proof is complete.
\end{proof}

\section{Lower bounds for a second order elliptic problem}
In this section, as an application of Theorems \ref{Lower_Bound_Theorem_Approximation_1} and \ref{Lower_Bound_Corollary_Approximation_1},  we will derive the lower bounds of the discretization error for a second order elliptic problem and the corresponding eigenpair problem by the finite element method.

Here we are concerned with the Poisson problem
\begin{equation}\label{Poisson}
\left\{
\begin{array}{rcl}
-\Delta u&=&f\ \ \ {\rm in}\ \Omega,\\
u&=&0\ \ \ {\rm on}\ \partial\Omega,
\end{array}
\right.
\end{equation}
and the corresponding eigenpair problem:

Find $(\lambda, u)$ such that $\|u\|_0=1$ and
\begin{equation}\label{Laplace_Eigenvalue}
\left\{
\begin{array}{rcl}
-\Delta u&=&\lambda u\ \ \ {\rm in}\ \Omega,\\
u&=&0\ \ \ \ \ {\rm on}\ \partial\Omega.
\end{array}
\right.
\end{equation}
Based on the partition $\mathcal{T}_h$ on $\bar{\Omega}$, we define a suitable
finite element space $V_h$ (conforming or nonconforming for the second order elliptic problem)
 with piecewise polynomials of degree less than $r$.

Then the finite element approximation of (\ref{Poisson}) consists of finding $u_h\in V_h$ such that
\begin{eqnarray}\label{Poisson_FEM}
a_h(u_h, v_h)&=&(f,v_h)\ \ \ \forall v_h\in V_h,
\end{eqnarray}
where $$a_h(u_h,v_h)=\sum_{K\in\mathcal{T}_h}\int_{K}\nabla u_h \nabla v_hdK.$$

From the standard error estimate theory of the finite element
method, it is known that the following upper bound of the
discretization error (see \cite{BrennerScott,Ciarlet}) holds
\begin{eqnarray}\label{upper_bound}
\|u-u_h\|_{\ell,p,h} &\leq& Ch^{s-\ell}\|u\|_{s,p} ,\ \ \ \ 0\leq\ell\leq 1,\ \ 0<s\leq r,
\end{eqnarray}
where $ 1<p<\infty$.

From Theorems \ref{Lower_Bound_Theorem_Approximation_1} and
\ref{Lower_Bound_Corollary_Approximation_1}, we state the following
lower bound results of the discretization error.
\begin{corollary}\label{Lower_bound_Solution_Theorem}
Assume there exist a subdomain $G\subset\subset\Omega$ such that $u\in W^{r+\delta,p}(G)$\ $(\delta>0)$ and
 a multi-index $\gamma$ with $|\gamma|=r$ such that $\|D^{\gamma}u\|_{0,p,G}>0$
and  $D^{\gamma}v_h=0$ for any $v_h\in V_h$.
If the family $\{\mathcal{T}_h\}$  of partitions is quasi-uniform, the finite element solution $u_h\in V_h$ in
(\ref{Poisson_FEM}) has the following lower bound of the discretization error
\begin{equation}\label{lower_convergence_1_1}
\|u-u_h\|_{j,p,h} \geq C_{10}h^{r-j},\ \ \ \ 0\leq j\leq r,
\end{equation}
where $1\leq p\leq \infty$,  $C_{10}$ is a positive constant dependent on $u$ and the error
estimate (\ref{upper_bound}) is optimal for $s=r$ and $0\leq \ell\leq 1$.
\end{corollary}
\begin{proof}
First we have the following property
\begin{eqnarray*}
\frac{\|u-u_h\|_{j,p,h}}{h^{r-j}}&\geq& \frac{\|u-u_h\|_{j,p,G,h}}{h^{r-j}}
\geq \inf_{v_h\in V_h} \frac{\|u-v_h\|_{j,p,G,h}}{h^{r-j}}.
\end{eqnarray*}
So the desired result (\ref{lower_convergence_1_1}) can be directly deduced by (\ref{lower_convergence_0_1}).
\end{proof}

\begin{remark}
The interior regularity result $u\in W^{r+\delta,p}(G)$ for a subdomain $G\subset\subset\Omega$ and $\delta>0$ for elliptic problem (\ref{Poisson}) can be obtained from
\cite[Theorem 8.10]{GilbargTrudinger} for the right-hand side $f$ smooth enough.
\end{remark}

\begin{corollary}\label{Lower_Bound_Corollary}
Assume there exist a subdomain $G\subset\subset\Omega$ such that $u\in W^{r+\delta,p}(G)$\ $(\delta>0)$ and
 a multi-index $\gamma$ with $|\gamma|=r$ such that $\|D^{\gamma}u\|_{0,p,G}>0$
and  $D^{\gamma}v_h=0$ for any $v_h\in V_h$.
If the family $\{\mathcal{T}_h\}$ of partitions is regular, the finite element solution $u_h\in V_h$ in
(\ref{Poisson_FEM}) has the following lower bound of the discretization error
\begin{equation}\label{lower_convergence_1_2}
\left(\sum_{K\in\mathcal{T}_h^G}h_K^{p(j-r)}
\big\|u-u_h\big\|_{j,p,K}^p\right)^{\frac{1}{p}}  \geq C_{11},\ \ \ \ 0\leq j\leq r,
\end{equation}
and
\begin{equation}\label{lower_convergence_1_3}
\left(\sum_{K\in\mathcal{T}_h^G}h_K^{p\big((j-r)+n(\frac{1}{p}-\frac{1}{q})\big)}
\big\|u-u_h\big\|_{j,q,K}^p\right)^{\frac{1}{p}} \geq C_{12},\ \ \ \ 0\leq j\leq r.
\end{equation}
where $1\leq p<\infty$, $1\leq q\leq \infty$ $(W^{r+\delta,p}(G)$ can be imbedded into $W^{j,q}(G))$, $C_{11}$
and $C_{12}$ are  positive constants dependent on $u$.
\end{corollary}

\begin{proof}
The proof can be given using the following property
\begin{eqnarray*}
\sum_{K\in\mathcal{T}_h^G}h_K^{p\big((j-r)+n(\frac{1}{p}-\frac{1}{q})\big)}
\big\|u-u_h\big\|_{j,q,K}^p \geq \inf_{v_h\in V_h}\sum_{K\in\mathcal{T}_h^G}h_K^{p\big((j-r)+n(\frac{1}{p}-\frac{1}{q})\big)}
\big\|u-v_h\big\|_{j,q,K}^p
\end{eqnarray*}
and Theorem \ref{Lower_Bound_Corollary_Approximation_1}.
\end{proof}

\begin{remark}
In \cite{ChenLi}, the lower bound of the discretization error by 
Wilson element has been analyzed under the conditions of the
rectangular partition and the regularity $u\in
W^{3,\infty}(\Omega)$.
\end{remark}

Now let us consider the lower bound analysis of the eigenpair problem (\ref{Laplace_Eigenvalue})
by the finite element method. The finite element approximation
$(\lambda_h,u_h)\in\mathcal{R}\times V_h$  of
(\ref{Laplace_Eigenvalue}) satisfies $\|u_h\|_0=1$ and
\begin{eqnarray}\label{Poisson_Eigenvalue_FEM}
a_h(u_h,v_h)&=&\lambda_h(u_h,v_h)\ \ \ \forall v_h\in V_h.
\end{eqnarray}
For the eigenfunction approximation $u_h$ in (\ref{Poisson_Eigenvalue_FEM}), the following
lower bound results hold.
\begin{corollary}\label{Lower_bound_Eigenvalue_Corollary}
Assume there exist
 a multi-index $\gamma$ with $|\gamma|=r$ such that $D^{\gamma}v_h=0$ for any $v_h\in V_h$.
If the family $\{\mathcal{T}_h\}$ of partitions  is quasi-uniform, the eigenpair approximation
$(\lambda_h, u_h)\in \mathcal{R}\times V_h$ in  (\ref{Poisson_Eigenvalue_FEM}) satisfies the following
lower bound of the discretization error
\begin{equation}\label{lower_convergence_2_1}
\|u-u_h\|_{j,p,h}\geq C_{13}h^{r-j},\ \ \ \ 0\leq j\leq r,
\end{equation}
where $1\leq p\leq \infty$ and $C_{13}$ is a positive constant dependent on $u$. 

Furthermore, if the family $\{\mathcal{T}_h\}$ of partitions  is regular, $(\lambda_h, u_h)$
has the following lower bounds of the discretization error
\begin{equation}\label{lower_convergence_2_2}
\left(\sum_{K\in\mathcal{T}_h^G}h_K^{p(j-r)}
\big\|u-u_h\big\|_{j,p,K}^p\right)^{\frac{1}{p}} \geq C_{14},\ \ \ \ 0\leq j\leq r.
\end{equation}
and
\begin{equation}\label{lower_convergence_2_3}
\left(\sum_{K\in\mathcal{T}_h^G}h_K^{p\big((j-r)+n(\frac{1}{p}-\frac{1}{q})\big)}
\big\|u-u_h\big\|_{j,q,K}^p\right)^{\frac{1}{p}} \geq C_{15},\ \ \ \ 0\leq j\leq r.
\end{equation}
where $1\leq p <\infty$, $1\leq q\leq \infty$,  $C_{14}$ and $C_{15}$ are  positive
constants dependent on $u$. 
\end{corollary}
\begin{proof}
First, it is easy to obtain that the eigenfunctions of problem (\ref{Laplace_Eigenvalue}) cannot
 be polynomial of bounded degree on any subdomain $G\subset\subset\Omega$. We prove this by reduction process.
 Assume the exact eigenfunction is a polynomial function and $u\in \mathcal{P}_{\ell}(G)$ for
 some integer $\ell>0$. Directly from the definition of problem (\ref{Laplace_Eigenvalue}), we have
\begin{eqnarray}
-\Delta^{\lceil\frac{\ell}{2}\rceil}u &=&(-1)^{\lceil\frac{\ell}{2}\rceil-1}
\lambda^{\lceil\frac{\ell}{2}\rceil}u,
\end{eqnarray}
where $\lceil\frac{\ell}{2}\rceil$ denotes the smallest integer not smaller than $\frac{\ell}{2}$. Since
$-\Delta^{\lceil\frac{\ell}{2}\rceil}u=0$, we have $u=0$ on $G$.
It means the exact eigenfunction cannot be polynomial of bounded degree and
has the following property
\begin{eqnarray*}
|u|_{r,p,G} > 0.
\end{eqnarray*}
The proof of this corollary can be obtained with the same argument as in the proof of
  Corollaries \ref{Lower_bound_Solution_Theorem} and \ref{Lower_Bound_Corollary}.
\end{proof}

Now, we present some conforming and nonconforming elements which yield the lower bound of
the discretization error with the help of  Corollaries \ref{Lower_bound_Solution_Theorem}, 
\ref{Lower_Bound_Corollary}, and \ref{Lower_bound_Eigenvalue_Corollary}.

In order to describe the results, we introduce the index set
 \begin{eqnarray}
Ind_{r}:=\big\{{\rm multi\ index}\ \alpha\ {\rm with}\ |\alpha|=r\big\}.
\end{eqnarray}
First we can obtain the lower bound results for the standard Lagrange type elements
\begin{eqnarray}\label{Lagrange_FEM}
V_h&=&\Big\{v_h|_K\in \mathcal{P}_{\ell}(K)\ {\rm or}\ \mathcal{Q}_{\ell}(K)\ \ \forall
K\in\mathcal{T}_h\Big\},
\end{eqnarray}
where $\mathcal{P}_{\ell}(K)$ denotes the space of polynomials  with degree not greater than $\ell$  and
$\mathcal{Q}_{\ell}(K)$ denotes the space of polynomials with degree not greater than $\ell$ in each variable.
>From  Corollaries \ref{Lower_bound_Solution_Theorem}, \ref{Lower_Bound_Corollary}, and \ref{Lower_bound_Eigenvalue_Corollary}, the lower bound
results in this section hold with $r=\ell+1$ and $\gamma\in Ind_{r}$ for $\mathcal{P}_{\ell}(K)$ case and
$r=\ell+1$ and $\gamma\in Ind_{r}\backslash Ind_{Q,\ell}$ for $\mathcal{Q}_{\ell}(K)$ case with $$Ind_{Q,\ell}:=\big\{{\rm multi\ index}\ \alpha\ {\rm with}\ \alpha_i\leq \ell\big\}.$$

Then it is also easy to check lower bound results for the following
four types of nonconforming elements Crouzeix-Raviart ($CR$), Extension of Crouzeix-Raviart ($ECR$),
$Q_1$ rotation ($Q_1^{\rm rot}$) and Extension of $Q_1$ rotation ($EQ_1^{\rm rot}$):
\begin{itemize}
\item
The $CR$ element space, proposed by Crouzeix and Raviart
\cite{CrouzeixRaviart}, is defined on simplicial partitions by
\begin{eqnarray*}
V_{h}&=&\Big\{v\in L^2(\Omega):\ v|_{K} \in \mathcal{P}_1(K), \nonumber\\
 &&\ \ \ \ \int_{F}v|_{K_{1}}ds =\int_{F}v|_{K_{2}}ds\ {\rm if}\ K_{1}\cap K_{2}=F\in \mathcal{E}_h\Big\}.
\end{eqnarray*}
The lower bound result holds with $r=2$ and $\gamma\in Ind_2$.
\item
The $ECR$ element space, proposed by Hu, Huang, and Lin \cite{HuHuangLin} and Lin, Xie, Luo, and Li \cite{LinXieLuoLi}, is defined on simplicial partitions by
\begin{eqnarray*}
V_{h}&=&\Big\{v\in L^2(\Omega):\ v|_{K} \in \mathcal{P}_{ECR}(K), \nonumber\\
 &&\ \ \ \ \int_{F}v|_{K_{1}}ds =\int_{F}v|_{K_{2}}ds\ {\rm if}\ K_{1}\cap K_{2}=F\in \mathcal{E}_h\Big\},
\end{eqnarray*}
where $\mathcal{P}_{ECR}(K)=\mathcal{P}_1(K)+{\rm span}\big\{\sum_{i=1}^nx_i^2\big\}$.
The lower bound result holds with $r=2$ and $\gamma$ with $\gamma_i=1, \gamma_j=1$, $1\leq i< j\leq n$.
\item
The $Q_{1}^{\rm rot}$ element space, proposed by Rannacher and Turek
\cite{RannacherTurek}, and Arbogast and Chen \cite{ArbogastChen}, is defined on $n$-dimensional block
partitions by
\begin{eqnarray*}
V_{h}&=&\Big\{v\in L^{2}(\Omega):v|_{K} \in Q_{\rm Rot}(K),\nonumber\\
&&\ \ \ \int_{F}v|_{K_{1}}ds =\int_{F}v|_{K_{2}}ds~~{\rm if}\
K_{1}\cap K_{2}=F\in\mathcal{E}_h\Big\},
\end{eqnarray*}
where $Q_{\rm Rot}(K)=\mathcal{P}_1(K)+{\rm span}\big\{x_i^2-x_{i+1}^2\ |\ 1\leq i\leq n-1\big\}$.
The lower bound result holds with $r=2$ and $\gamma$ with $\gamma_i=1, \gamma_j=1$, $1\leq i< j\leq n$.
\item
The $EQ_{1}^{\rm rot}$ element space, proposed by Lin, Tobiska, and Zhou \cite{LinTobiskaZhou}, is defined on $n$-dimensional block partitions by
\begin{eqnarray*}
V_{h}&=&\Big\{v\in L^{2}(\Omega):v|_{K} \in Q_{\rm ERot}(K),\nonumber\\
&&\ \ \ \int_{F}v|_{K_{1}}ds =\int_{F}v|_{K_{2}}ds~~{\rm if}\
K_{1}\cap K_{2}=F\in\mathcal{E}_h\Big\},
\end{eqnarray*}
where $Q_{\rm Rot}(K)=\mathcal{P}_1(K)+{\rm span}\big\{x_i^2\ |\ 1\leq i\leq n\big\}$.
The lower bound result holds with $r=2$ and $\gamma$ with $\gamma_i=1, \gamma_j=1$, $1\leq i< j\leq n$.
\end{itemize}
All lower bounds of the above four examples are  sharp if the solution is smooth enough.
For other types of finite elements, we could also obtain the lower bound results with
the corresponding $r$ and $\gamma$ as in this section.

\section{Lower bounds for $2m$-th order elliptic problem}
We consider the similar lower bounds of the discretization error for $2m$-th order elliptic
 problem and the corresponding eigenpair problem by the finite element method. Actually, this is a
 natural generalization of the results in Section 3.

 The $2m$-th order Dirichlet elliptic problem for a given integer $m\geq 1$ is defined as
\begin{equation}\label{2m_Problem}
\left\{
\begin{array}{rcl}
(-1)^m\Delta^{m}u&=&f\ \ \ {\rm in}\ \Omega,\\
\frac{\partial^ju}{\partial^j\mathbf \nu}&=&0\ \ \ {\rm on}\ \partial\Omega\ {\rm and}\ 0\leq j\leq m-1,
\end{array}
\right.
\end{equation}
where $\nu$ denotes the unit outer normal.
The corresponding weak form of problem (\ref{2m_Problem}) is to seek $u\in H_0^m(\Omega)$ such that
\begin{eqnarray}\label{2m_Problem_weak}
a(u,v)&=&(f,v)\ \ \ \ \forall v\in H_0^m(\Omega),
\end{eqnarray}
where
$$a(u,v)=\int_{\Omega}\sum_{|\alpha|=m}D^{\alpha} u D^{\alpha} v\;d\Omega.$$

Based on the partition $\mathcal{T}_h$ of $\bar{\Omega}$, we build a suitable
finite element space $V_h$ (conforming or nonconforming for $2m$-th order elliptic problem)
 with piecewise polynomial of degree less than $r$. The finite element approximation of
 (\ref{2m_Problem}) is to seek $u_h\in V_h$ satisfying
\begin{eqnarray}\label{2m_th_FEM}
a_h(u_h, v_h)&=&(f,v_h)\ \ \ \forall v_h\in V_h,
\end{eqnarray}
where $$a_h(u_h,v_h)=\sum_{K\in\mathcal{T}_h}\int_{K}\sum_{|\alpha|=m}D^{\alpha} u_h D^{\alpha} v_hdK.$$

We also consider the corresponding $2m$-th order elliptic eigenpair problem:

Find $(\lambda,u)\in\mathcal{R}\times H_0^m(\Omega)$ such that $\|u\|_0=1$ and
\begin{eqnarray}\label{2m_Eigenvalue}
a(u,v)&=&\lambda (u,v)\ \ \ \ \forall v\in H_0^m(\Omega).
\end{eqnarray}

In this section, we assume that the following upper bound of the discretization error holds
\begin{eqnarray}\label{Upper_Bound_m_Problem}
\|u-u_h\|_{m,h}&\leq &Ch^{s-m}\|u\|_{s},\ \ \ 0<s\leq r.
\end{eqnarray}

Similarly to Corollaries  \ref{Lower_bound_Solution_Theorem} and
 \ref{Lower_Bound_Corollary}, the finite element approximation $u_h$
possesses the following lower bound results.
\begin{corollary}\label{Lower_Bound_Theorem_m_th}
Assume there exist a subdomain $G\subset\subset\Omega$ such that $u\in W^{r+\delta,p}(G)$\ $(\delta>0)$ and
 a multi-index $\gamma$ with $|\gamma|=r$ such that $\|D^{\gamma}u\|_{0,p,G}>0$
and  $D^{\gamma}v_h=0$ for any $v_h\in V_h$.
If the family $\{\mathcal{T}_h\}$ of partitions is quasi-uniform, the finite element solution $u_h\in V_h$ in
(\ref{2m_th_FEM}) has the following lower bound of the discretization error
\begin{equation}\label{lower_convergence_2m_th_3_1}
\|u-u_h\|_{j,p,h} \geq C_{16}h^{r-j},\ \ \ \ 0\leq j\leq r,
\end{equation}
where $1\leq p\leq \infty$ and  $C_{16}$ is a positive constant dependent on $u$ and
 the error estimate (\ref{Upper_Bound_m_Problem}) is optimal for $s=r$.
\end{corollary}
\begin{proof}
First we have the following property
\begin{eqnarray*}
\frac{\|u-u_h\|_{j,p,h}}{h^{r-j}}&\geq& \frac{\|u-u_h\|_{j,p,G,h}}{h^{r-j}}\geq\inf_{v_h\in V_h} \frac{\|u-v_h\|_{j,p,G,h}}{h^{r-j}}.
\end{eqnarray*}
So the desired result (\ref{lower_convergence_2m_th_3_1}) can be directly deduced by (\ref{lower_convergence_0_1}).
\end{proof}
\begin{remark}
The interior regularity result $u\in W^{r+\delta,p}(G)$ for a subdomain $G\subset\subset\Omega$ and $\delta>0$ for problem (\ref{2m_Problem}) can be obtained from
\cite[Theorem 7.1.2]{Grisvard} for the right-hand side $f$ smooth enough.
\end{remark}

\begin{corollary}\label{Lower_bound_2m_Solution_Corollary}
Assume there exist a subdomain $G\subset\subset\Omega$ such that $u\in W^{r+\delta,p}(G)$\ $(\delta>0)$ and
 a multi-index $\gamma$ with $|\gamma|=r$ such that $\|D^{\gamma}u\|_{0,p,G}>0$
and  $D^{\gamma}v_h=0$ for any $v_h\in V_h$.
If the family $\{\mathcal{T}_h\}$ of partitions  is regular, the finite element solution $u_h\in V_h$ in
(\ref{Poisson_FEM}) has the following lower bound of the discretization error
\begin{equation}\label{lower_convergence_2m_th_3_2}
\left(\sum_{K\in\mathcal{T}_h^G}h_K^{p(j-r)}
\big\|u-u_h\big\|_{j,p,K}^p\right)^{\frac{1}{p}} \geq  C_{17},\ \ \ \ 0\leq j\leq r,
\end{equation}
and
\begin{equation}\label{lower_convergence_2m_th_3_3}
\left(\sum_{K\in\mathcal{T}_h^G}h_K^{p\big((j-r)+n(\frac{1}{p}-\frac{1}{q})\big)}
\big\|u-u_h\big\|_{j,q,K}^p\right)^{\frac{1}{p}} \geq C_{18},\ \ \ \ 0\leq j\leq r,
\end{equation}
where $1\leq p<\infty$, $1\leq q\leq \infty$ $(W^{r+\delta,p}(G)$ can be imbedded into $W^{j,q}(G))$, $C_{17}$
  and $C_{18}$ are  positive constants
 independent of mesh size $h_G$, but dependent on $u$.
\end{corollary}

Now we introduce the corresponding lower bound analysis of the eigenpair problem (\ref{2m_Eigenvalue}).
We define the corresponding discrete eigenpair problem in the finite element space:

Find $(\lambda_h,u_h)\in\mathcal{R}\times V_h$ such that $\|u_h\|_0=1$ and
\begin{eqnarray}\label{2m_Eigenvalue_FEM}
a_h(u_h,v_h)&=&\lambda_h (u_h,v_h)\ \ \ \ \forall v_h\in V_h.
\end{eqnarray}
The eigenfunction approximation $u_h$ in (\ref{2m_Eigenvalue_FEM}) also gives the lower bound results as
follows.
\begin{corollary}\label{Lower_bound_2m_Eigenvalue_Corollary}
Assume there exist a multi-index $\gamma$ with $|\gamma|=r$ such that $D^{\gamma}v_h=0$ for any $v_h\in V_h$.
If the family  $\{\mathcal{T}_h\}$ of partitions is quasi-uniform, the eigenpair approximation
$(\lambda_h, u_h)\in \mathcal{R}\times V_h$ in  (\ref{2m_Eigenvalue_FEM}) yield the following
lower bound of the discretization error
\begin{equation}\label{lower_convergence_2m_th_Eigenvalue_1}
{\|u-u_h\|_{j,p,h}} \geq C_{19}{h^{r-j}},\ \ \ \ 0\leq j\leq r,
\end{equation}
where $1\leq p\leq \infty$ and  $C_{19}$ is a positive constant independent of mesh size.

Furthermore, if the family $\{\mathcal{T}_h\}$  of partitions is only regular, $(\lambda_h, u_h)$
has the following lower bounds of the discretization error
\begin{equation}\label{lower_convergence_2m_th_Eigenvalue_2}
\left(\sum_{K\in\mathcal{T}_h^G}h_K^{p(j-r)}
\big\|u-u_h\big\|_{j,p,K}^p\right)^{\frac{1}{p}} \geq C_{20},\ \ \ \ 0\leq j\leq r,
\end{equation}
and
\begin{equation}\label{lower_convergence_2m_th_Eigenvalue_3}
\left(\sum_{K\in\mathcal{T}_h^G}h_K^{p\big((j-r)+n(\frac{1}{p}-\frac{1}{q})\big)}
\big\|u-u_h\big\|_{j,q,K}^p\right)^{\frac{1}{p}} \geq C_{21},\ \ \ \ 0\leq j\leq r,
\end{equation}
where $1\leq  p < \infty$, $1\leq q\leq \infty$, $C_{20}$ and $C_{21}$ are
positive constants dependent on $u$. 
\end{corollary}
\begin{proof}
Similarly to Corollary \ref{Lower_bound_Eigenvalue_Corollary}, it is easy to obtain that
the eigenfunction of problem (\ref{2m_Eigenvalue}) cannot
 be a polynomial function of bounded degree on any subdomain $G\subset\subset\Omega$.
 It means the eigenfunction has the following property
\begin{eqnarray*}
|u|_{r,p,G}> 0.
\end{eqnarray*}
Then the proof can be obtained with the same argument as in the proof of Corollaries
\ref{Lower_Bound_Theorem_m_th} and  \ref{Lower_bound_2m_Solution_Corollary}.
\end{proof}

Now, we give some types of conforming and nonconforming elements which can produce lower bound of
the discretization error with the help of Corollaries \ref{Lower_Bound_Theorem_m_th},
\ref{Lower_bound_2m_Solution_Corollary}, and \ref{Lower_bound_2m_Eigenvalue_Corollary}.

First we would like to remind that for the two-dimensional case ($n=2$) there exist elements such as the Argyris and Hsieh-Clough-Tocher elements \cite{Ciarlet} and etc., which yield lower bound results from Corollaries \ref{Lower_Bound_Theorem_m_th}, \ref{Lower_bound_2m_Solution_Corollary}, and \ref{Lower_bound_2m_Eigenvalue_Corollary} for the biharmonic problem. The lower bound results
in this section hold for the Argyris element with $m=2$, $r=6$, $\gamma\in Ind_6$ and the Hsieh-Clough-Tocher
element with $m=2$, $r=4$, $\gamma\in Ind_4$, respectively.

Furthermore, we consider a family of nonconforming element named by MWX  proposed by Wang and Xu \cite{WangXu}
and apply it to the $2m$-th order elliptic problem and the corresponding eigenpair problem under consideration.
The MWX element with $n\geq m\geq 1$ is the triple  $(K,\mathcal{P}_K, D_K)$,
where $K$ is a $n$-simplex and $\mathcal{P}_K=\mathcal{P}_m(K)$. For a description of the set
$D_K$ of degrees of freedom, see \cite{WangXu}. 

In order to understand this element, we list some special cases as in \cite{WangXu} for $1\leq m\leq 3$.
If $m=1$ and $n=1$, we obtain the well-known conforming linear elements. This is the only conforming
element in this family of elements. For $m=1$ and $n\geq 2$, we obtain the
well-known nonconforming linear element ($CR$).
If $m=2$, we recover the well-known nonconforming Morley element for $n=2$ and its generalization to
$n\geq 2$ (see Wang and Xu \cite{WangXu_4th}). For $m=3$ and $n=3$, we obtain a cubic
element on a simplex that has $20$ degrees of freedom.

Based on the above description of MWX element, we know that
 $$|v_h|_{1+m,p,h}\equiv 0\ \ \ \ \ \forall v_h\in V_h.$$
Then with the help of Corollaries \ref{Lower_Bound_Theorem_m_th},
\ref{Lower_bound_2m_Solution_Corollary}, and \ref{Lower_bound_2m_Eigenvalue_Corollary},
we get the lower bound results in this section with $r=m+1$ and $\gamma\in Ind_{r}$.

We can also obtain the lower bound results in this section for other types of elements with
suitable $r$ and $\gamma$ for the $2m$-th order elliptic problem (\ref{2m_Problem}) and the
 corresponding eigenvalue problem (\ref{2m_Eigenvalue}).

\section{Concluding remarks}
In this paper, a type of lower bound results of the error by piecewise polynomial approximation
is presented. As  applications,  we give the lower bounds of the discretization error for second order
 elliptic and $2m$-th order elliptic problem by finite element methods. From the analysis, the idea
 and methods here can be extended to other problems and numerical methods which are based on the piecewise
 polynomial approximation.

\vskip0.5cm

\noindent{\bf Acknowledgements.} The authors wish to thank Professor
Michal K\v{r}\'{i}\v{z}ek for helpful discussions.

\end{document}